\documentclass{amsart}

\usepackage{amscd,amsfonts,amsmath,amssymb,amstext,amsthm}
\usepackage{amsrefs}
\usepackage{enumerate,paralist}
\usepackage{graphicx,tikz}
	\usetikzlibrary{backgrounds,calc,fit,matrix}	
\usepackage[mathscr]{euscript}
\usepackage[all,pdf]{xy}
	\CompileMatrices
	
\theoremstyle{plain}
\newtheorem{thm}{Theorem}
\newtheorem*{thm*}{Theorem}
\newtheorem{prop}{Proposition}

\theoremstyle{definition}
\newtheorem{dfn}{Definition}
\newtheorem{que}{Question}
\newtheorem{ex}{Example}

\newcommand{\Aff}{\mathbb A} 
\newcommand{\Prj}{\mathbb P} 
\newcommand{\transpose}{^{\rm T}} 
\newcommand{\vect}[1]{\mathbf{#1}}
\newcommand{\Z}{\mathbb Z}

\DeclareMathOperator{\pr}{\frak B} 
\DeclareMathOperator{\codim}{codim}
\DeclareMathOperator{\col}{col}
\DeclareMathOperator{\GL}{GL}
\DeclareMathOperator{\Graph}{Graph}
\DeclareMathOperator{\Grass}{Grass}
\DeclareMathOperator{\I}{I} 
\DeclareMathOperator{\row}{row}
\DeclareMathOperator{\Spec}{Spec}
\DeclareMathOperator{\Var}{\mathscr V} 

\begin{document}

\title{Principal Minor Ideals and Rank Restrictions on their Vanishing Sets}
\author{Ashley K. Wheeler}
\address{Department of Mathematical Sciences \\
	University of Arkansas \\
	Fayetteville, AR 72701
	}
\email{ashleykw@uark.edu}
\thanks{This work partially partially supported by NSF grant 0943832.}

\begin{abstract} All matrices we consider have entries in a fixed algebraically closed field $K$.  A minor of a square matrix is \emph{principal} means it is defined by the same row and column indices.  We study the ideal generated by size $t$ principal minors of a generic matrix, and restrict our attention to locally closed subsets of its vanishing set, given by matrices of a fixed rank.  The main result is a computation of the dimension of the locally closed set of $n\times n$ rank $n-2$ matrices whose size $n-2$ principal minors vanish; this set has dimension $n^2-n-4$. \end{abstract}

\maketitle


\section{Introduction}\label{sec:intro}

Given a generic $n\times n$ matrix $X$, and $K[X]$ the polynomial ring in entries $x_{ij}$ of $X$ over some algebraically closed field $K$, we study the ideals $\pr_t=\pr_t(X)$, generated by the size $t$ principal minors of $X$.  Historically, various ideals defined using generic matrices have been of great interest to algebraistis -- such examples include the determinantal ideals (see ~\cites{room, eagon_thesis, eagon+northcott, hochster+eagon, svanes, hochster+huneke/94_2}), due to their connection to invariant theory (as in \cite{deconcini+procesi76}) and the Pfaffian ideals (see ~\cites{kleppe, jozefiak+pragacz79, kleppe+laksov80, pragacz81, denegri+gorla11}), whose study is often inspired by the result from \cite{buchsbaum+eisenbud}, as well as their connection to invariant theory.  In developing their generalized version of the Principal Minor Theorem, Kodiyalam, Lam, and Swan (\cite{kodiyalam+lam+swan}) reveal a contrast between the principal minor ideals and the Pfaffian ideals: while the Pfaffian ideals, like the determinantal ideals, satisfy a chain condition according to rank, the principal minor ideals do not.  Furthermore, in \cite{wheeler/14} it is shown that, unlike determinantal ideals and Pfaffian ideals, principal minor ideals are not, in general, Cohen-Macaulay.

Principal minors arise in many other contexts -- see, for example, ~\cite{stouffer24, stouffer28, griffin+tsatsomeros05, oeding11_1, oeding11_2}.  The most direct study of principal minors is in \cite{wheeler/14}.  There, it is shown the algebraic set $\Var(\pr_{n-1})$ has two components: one given by the determinantal ideal $\I_{n-1}$ and the other given by a height $n$ ideal, $\frak Q_{n-1}$, that is the contraction to $K[X]$ of the kernel of the map 
\begin{align*}
K[X]\left[\frac{1}{\det\,X}\right] &\to K[X]\left[\frac{1}{\det\,X}\right]/\pr_1 \\
X&\to X^{-1}.
\end{align*}
When $n=4$, $\pr_{n-1}$ is reduced and as a consequence, $\I_{n-1}$ and $\frak Q_{n-1}$ are linked in that case.  Identifying the components for $\Var(\pr_{n-1})$ relies on another result within that paper, that if $\mathscr Y_{n,r,t}$ denotes the locally closed set of $\Var(\pr_t)$ consisting of rank $r$ matrices, then for all $n,t$, 
\[\mathscr Y_{n,n,t}\cong\mathscr Y_{n,n,n-t}\]
as schemes.

This paper is organized as follows: Section \ref{sec:matrices} gives the necessary preliminaries for the remainder of the paper.  We focus on the components of $\Var(\pr_t)$ by restricting to the locally closed subsets $\mathscr Y_{n,r,t}$, consisting of matrices of rank exactly $r$.  Our main result, given in Section \ref{sec:n-2}, is a computation of the dimension of $\mathscr Y_{n,n-2,n-2}$.

\begin{thm*}[\ref{thm:n-2,n-2}, Section \ref{sec:conseq}] The locally closed set of $n\times n$ rank $n-2$ matrices in $\Spec K[X]$, whose size $n-2$ principal minors vanish, has dimension $n^2-4-n$.  \end{thm*}

In studying the components of $\mathscr Y_{n,n-2,n-2}$ we define a bundle map $\Theta$ (see Equation (\ref{eq:Theta}), Section \ref{sec:matrices}) that reduces the problem to studying pairs of subsets in $\Grass(n-2,n)$.  The technique in proving Theorem \ref{thm:n-2,n-2} is as follows: Given a point in the Grassmannian, we encode exactly which of its Pl\"ucker coordinates do and do not vanish in a simple graph.  Such graphs are called \emph{permissible} (see Section \ref{sec:pluckersToGraphs}).  We then define the notion of a permissible subvariety of the Grassmannian, along with its corresponding graph.  We prove and then use the properties of permissible graphs to compute the dimension of $\mathscr Y_{n,n-2,n-2}$.  

In Section \ref{sec:n-3} we suggest a natural extension of the techniques from Section \ref{sec:pluckersToGraphs} to the locally closed sets $\mathscr Y_{n,n-3,n-3}$.  More generally, the structure of $\mathscr Y_{n,t,t}$ turns out to be of great interest in its own right, in fact leading to questions that are NP-hard (see \cite{ford} the result cited therein from \cite{shor}), though such questions are beyond the scope of this paper.  

Finally, in Section \ref{sec:n=5} we go through the results of Section \ref{sec:conseq} for the case where $n=5$.  

\section{Preliminaries}\label{sec:matrices}

As a tool in studying the components of $\Var(\pr_t)$, we study the components of the locally closed sets $\mathscr Y_{n,r,t}\subset\Var(\pr_t)$, the $n\times n$ matrices of rank $r$ whose size $t$ principal minors vanish.  As $r$ ranges, the irreducible components of $\mathscr Y_{n,r,t}$ cover $\Var(\pr_t)$, hence so do their closures.  There are finitely many such closures; whenever any closed set is a finite union of closed varieties, we can find its irreducible components by making the union irredundant.  In other words, the components of $\Var(\pr_t)$ can be found by first finding the closures of all the components of the $\mathscr Y_{n,r,t}$, as $r$ varies, and then omitting the ones that are not maximal in the family.   

To study $\mathscr Y_{n,r,t}$ we observe and use its relationship to Grassmann varieties in the following way:  For any matrix $A$, let $\col A$ and $\row A$ denote, respectively, the column space and row space of $A$.  The Grassmann variety, or Grassmannian, $\Grass(r,n)\subset\Prj^{\binom{n}{r}-1}$, is the projective variety whose points are in bijection with the $r$-dimensional vector spaces of $K^n$.  As shorthand, we shall put $\mathscr G=\Grass(r,n)$ for fixed $r\leq n$.  Under the Pl\"ucker embedding, write 
\[\vect g=\left[g_{\{1,\dots,r\}}:\cdots:g_{\underbar i}:\cdots:g_{\{n-r+1,\dots,n\}}\right]\]
to denote a point in $\mathscr G$, so the Pl\"ucker coordinates are indexed by sets $\underbar i\subset\{1,\dots,n\}$.    

Suppose $\col\,B\in\mathscr G$ for some $n\times r$ matrix $B$ of full rank.  Letting $B(\underbar i)$ denote the submatrix of $B$ consisting of the rows indexed by elements in $\underbar i\subset\{1,\dots,n\}$, there exists such an $\underbar i=\{i_1,\dots,i_r\}$ satisfying $\det\,B(\underbar i)\neq 0$.  Thus we may perform row operations on $B$ to get a unique matrix $B'$, such that $\col\,B'=\col\,B$ and $B'(\underbar i)=\vect I_{r\times r}$, the $r\times r$ identity matrix.  We shall call $B'$ the \emph{normalized} form of $B$, or \emph{normalization}, with respect to $\underbar i$.  Analogously, if $\row\,C\in\mathscr G$ for some $r\times n$ matrix $C$, then we define the normalized form $C'$ of $C$ with respect to a set of column indices, $\underbar j=\{j_1,\dots,j_r\}$, provided the submatrix, $C(\underbar j)$ of $C$, consisting of the columns indexed by elements in $\underbar j$,  is non-singular.  The reason for considering both $n\times r$ and $r\times n$ representatives of $\Grass(r,n)$ is motivated by the following important observation:   

\begin{prop}\label{prop:GrassMap} Let $\mathscr Z_{n,r}\subset\Spec\,K[X]$ denote the set of $n\times n$ matrices of rank exactly $r$.  Then
\begin{equation}\label{eq:Theta}
\begin{split}
\Theta:\mathscr Z_{n,r} & \to\mathscr G\times\mathscr G \\
A & \mapsto\left(\col A,\row A\right).
\end{split}
\end{equation}
is a bundle map whose fibres are each isomorphic to $\GL(r,K)$. \end{prop}

\begin{proof} The sets where a specified Pl\"ucker coordinate does not vanish give an affine open cover of $\mathscr G$, and hence of $\mathscr G\times\mathscr G$.  Explicitly, $\mathscr G$ is covered by the open sets 
\[\mathscr G_{\underbar i}=\left\{\left[\cdots:g_{\underbar i}:\cdots\right]\in\mathscr G\,|\,g_{\underbar i}\neq 0\right\}\cong\Aff^{r(n-r)}.\]   

We will show, for each open set $\mathscr G_{\underbar i}\times\mathscr G_{\underbar j}$, that the diagram
\[\xymatrix{
\Theta^{-1}\left(\mathscr G_{\underbar i}\times\mathscr G_{\underbar j}\right) \ar@{->>}[d]_{\Theta} \ar[r]^{\cong} &
	\mathscr G_{\underbar i}\times\mathscr G_{\underbar j}\times\GL(r,K) \ar@{->>}[d]^{\pi} \\
\mathscr G_{\underbar i}\times\mathscr G_{\underbar j} \ar[r]^{=} & \mathscr G_{\underbar i}\times\mathscr G_{\underbar j}
}\]
commutes, where $\pi$ is the projection map.  The preimage of $\mathscr G_{\underbar i}\times\mathscr G_{\underbar j}$ consists of matrices $A\in\mathscr Z_{n,r}$ that factor 
\begin{equation}\label{eq:norm}
A=BC=B'A(\underbar i;\underbar j)C',
\end{equation}
where $B'$ is the normalization with respect to $\underbar i$ of the $n\times r$ matrix $B$, $C'$ is the normalization with respect to $\underbar j$ of the $r\times n$ matrix $C$, and $A(\underbar i;\underbar j)$ is the submatrix of $A$ consisting of its $\underbar i$-rows and $\underbar j$-columns.  By uniqueness of the normalizations, given fixed $\underbar i,\underbar j$, such pairs of matrices $(B',C')$ are in bijection with points in $\mathscr G_{\underbar i}\times\mathscr G_{\underbar j}$.  For any fixed pair $(B',C')$, the set of all possibilities for $A(\underbar i;\underbar j)$ that satisfy Equation (\ref{eq:norm}) is in bijection with $\GL(r,K)$.  The maps are clearly regular.  
\end{proof}

\subsection{Irreducible Components of $\mathscr Y_{n,t,t}$}\label{sec:r=t}

Fix $r=t$ and $\mathscr G=\Grass(t,n)$.  Suppose we factor $A\in\mathscr Y_{n,t,t}$ as in Equation (\ref{eq:norm}).  Then we must have $\underbar i\neq\underbar j$.  Furthermore, requiring the size $t$ principal minors of $A$ to vanish means, equivalently, the diagonal entries of the exterior power matrix $\wedge^tA$ must vanish.  Write
\[\wedge^tA=\left(\wedge^tB\right)\cdot\left(\wedge^tA(\underbar i;\underbar j)\right)\cdot\left(\wedge^tC\right).\]
Each of the factors $\wedge^tB,\wedge^tC$ are, respectively, column and row vectors, while $\wedge^tA(\underbar i;\underbar j)$ is a (non-zero) scalar.  

Up to sign, the Pl\"ucker coordinates of the column space of any $n\times t$ matrix of full rank, $t\leq n$, are the coordinates of the exterior product of the columns with respect to the basis $\{\vect e_{i_1}\wedge\cdots\wedge\vect e_{i_t}\,|\,1\leq i_1<\cdots<i_t\leq n\}$, where for $i\in\{1,\dots,n\}$, $\vect e_i$ is the standard basis vector in $K^n$ given by the $i$th row of the identity matrix; the analogous statement holds for the row space of any $t\times n$ matrix of full rank.  It follows that the principal $t$-minors of a matrix $A\in\mathscr Z_{n,t}$ vanish if and only if the component-wise product of $\wedge^tB$ and $\wedge^tC$ is zero.  

The inclusion $\mathscr Y_{n,t,t}\hookrightarrow\mathscr{Z}_{n,t}$ induces, via Equation (\ref{eq:Theta}), a bundle map:
\begin{equation}\label{eq:bundleDiagram}
\begin{gathered}
\xymatrix{
\mathscr Y_{n,t,t} \ar @{^{(}->}[r] \ar @{->>}[d] & \mathscr Z_{n,t} \ar @{->>}[d]^{\Theta} \\
\Theta(\mathscr Y_{n,t,t}) \ar @{^{(}->}[r] & \mathscr G\times\mathscr G
}
\end{gathered}
\end{equation}
From (\ref{eq:bundleDiagram}), let $\mathscr H\subseteq\mathscr G\times\mathscr G$ denote the closed set consisting of pairs $(\vect g,\vect h)$ where for each index $\underbar i$, either $g_{\underbar i}$ or $h_{\underbar i}$ vanishes.  Then $\mathscr Y_{n,t,t}=\Theta^{-1}(\mathscr H)$, and since $\Theta$ is a bundle map, the components of $\mathscr Y_{n,t,t}$ correspond bijectively to the components of $\mathscr H$.  

To get an irreducible component of $\mathscr H$, we must partition the set of indices for the Pl\"ucker coordinates into two sets, $\underbar I,\underbar J$.  Let $\Var(\underbar I),\Var(\underbar J)$ denote the respective closed subsets of $\mathscr G$ defined by the vanishing of Pl\"ucker coordinates respectively indexed by $\underbar I,\underbar J$.  Each component of $\mathscr H$ must be a component of $\Var(\underbar I)\times\Var(\underbar J)$ for some partition $\underbar I\coprod\underbar J$ and every component of $\Var(\underbar I)\times\Var(\underbar J)$ arises as the product of a component of $\Var(\underbar I)$ and a component of $\Var(\underbar J)$.  Our goal is to consider all such partitions $\underbar I,\underbar J$, and then for each component $\mathscr C$ of $\Var(\underbar I)$ and each component $\mathscr D$ of $\Var(\underbar J)$, we shall consider the irreducible set $\mathscr C\times\mathscr D$.  The components of $\mathscr H$ are the maximal such sets $\mathscr C\times\mathscr D$, and their inverse images under the bundle map $\Theta$ give the irreducible components of $\mathscr Y_{n,t,t}$.

\subsection{Group Actions on $\pr_t$}\label{sec:obs}

In order to justify some of the arguments of preserving generality in what follows, we briefly remark on the group actions on $\pr_t$ which leave it invariant.  The most useful is given by the symmetric group of degree $n$.  Suppose $\tau$ is a size $n$ permutation matrix, $\tau\transpose$ its transpose. The action $X\mapsto\tau X\tau\transpose$ performs the same permutation on the rows of $X$ as it does the columns, hence preserves $\pr_t$.  The obvious action $\Z/2\Z$ given by $X\mapsto X\transpose$ preserves $\pr_t$ as well.  Finally, $\pr_t$ is unaffected by scalar multiplication, i.e., the action of $\GL(1,K)\cong K^{\times}$ on each row and each column of $X$.

\section{Principal $(n-2)$-Minors Case}\label{sec:n-2}

The problem of understanding all the components of the various sets $\Var(\underbar I)$, as described in Section \ref{sec:r=t}, is known to be extremely hard (as mentioned in \cites{ford,shor}).  However, we shall be able to understand the situation completely when $r=t=n-2$.  Since an $(n-2)$-minor of an $n\times(n-2)$ matrix is determined by the two rows that are not used we can alternatively describe $\underbar I$ using a collection of 2 element subsets of $\{1,\dots,n\}$, and then encode that data into a simple graph as follows: label the vertices $\{v_1,\dots,v_n\}$ and draw an edge between two vertices $v_{i_1},v_{i_2}$ if and only if $\{1,\dots,n\}\setminus\{i_1,i_2\}$ is not in any set in $\underbar I$.  We can associate an analogous graph to a set $\underbar J$ of indices corresponding to minors of an $(n-2)\times n$ matrix.

We shall see that a component of a closed set of the form $\Var(\underbar I)$ can also be encoded into a graph, and we will give a condition on its graph that is equivalent to irreducibility.  We will then classify the minimal pairs of graphs that together cover the complete graph of order $n$ and such that each graph corresponds to an irreducible closed set in $\Grass(n-2,n)$.  In Section \ref{sec:n=5} we work through an explicit case, $n=5$.  Throughout this section $\mathscr G$ shall denote the Grassmann variety, $\Grass(n-2,n)$, under the Pl\"ucker embedding.  We also assert $n\geq 3$.  

\subsection{Basic Graph Theory Definitions}

Before proceeding, we recall some basic graph theory notions:  A graph is \emph{order} $n$ means it has $n$ vertices.  Two vertices in a graph are \emph{adjacent} means there is an edge joining them.  The \emph{degree} of a vertex $v$ is the number of edges incident to it, where a \emph{loop}, an edge joining $v$ to itself, counts as two edges.  A vertex is \emph{isolated} means it has no edges. A graph is \emph{simple} means every edge joins exactly two vertices (i.e., there are no loops) and any two vertices are joined by at most one edge (i.e., there are no \emph{parallel edges}).  All graphs to which we refer from now on are simple.  

Suppose $G$ is a graph of order $n$.  A vertex in $G$ is \emph{dominating} means it has degree $n-1$, i.e., it is adjacent to every other vertex in $G$.  $G$ is \emph{complete} means all of its vertices are dominating.  For any $a$, we use $K_a$ to denote the complete graph of order $a$.  A \emph{clique} is a complete subgraph $K_a\subseteq G$.  A \emph{maximal clique} of order $a$ is a subgraph $K_a\subseteq G$ that is not properly contained in any clique.  A collection of simple graphs of order $n$ is a \emph{cover} (or \emph{covering}) means the union of their edges gives $K_n$.  The \emph{complement} $H$ of $G$ is the unique graph that, with $G$, gives a cover.    

\subsection{Pl\"ucker Coordinates to Graphs}\label{sec:pluckersToGraphs}
 
Let $G=\Graph(\vect g)$, for a fixed point $\vect g\in\mathscr G$, where we construct $\Graph(\vect g)$ as follows:  $G$ consists of vertices $v_1,\dots,v_n$, where $v_{i_1}$ and $v_{i_2}$ are adjacent if and only if the Pl\"ucker coordinate $g_{\underbar i}$, where $\underbar i=\{1,\dots,n\}\setminus\{i_1,i_2\}$, vanishes.  If $A$ is a matrix whose column span (resp., row span) is $\vect g$, then we write $\Graph(A)=\Graph(\col\,A)$ (resp., $\Graph(A)=\Graph(\row\,A)$).  Apparently such graphs are not in bijection with the simple graphs of order $n$.  

\begin{dfn}\label{def:perm} A graph $G$ of order $n$ is \emph{permissible} means 
\begin{enumerate}[(a)]
\item $G$ has at most $\binom{n}{n-2}-1$ edges, i.e., $G$ is not complete, and 
\item\label{dfn:omit-dom} its subgraph obtained by omitting all dominating vertices is a disjoint union of maximal cliques.
\end{enumerate}
\end{dfn}

Figure \ref{fig:permissGraphs} shows examples of graphs which are permissible and Figure \ref{fig:antiPermiss} shows examples of graphs which are not permissible.  Condition (\ref{dfn:omit-dom}) in Definition \ref{def:perm} gives a way to construct permissible graphs; the following is an equivalent condition that is useful in identifying permissible graphs.

\begin{prop}\label{prop:equiv-perm} Suppose $G$ is a non-complete graph of order $n$.  Then the following are equivalent:
\begin{enumerate}[(i)]
\item\label{prop:equiv-perm1} The subgraph of $G$ obtained by omitting all dominating vertices is a disjoint union of cliques.
\item\label{prop:equiv-perm2} A vertex $v\in G$ has degree $d\neq n-1$ if and only if $v$ is part of a maximal clique of order $d$. 
\end{enumerate}
\end{prop}

\begin{proof} We shall show (\ref{prop:equiv-perm2}) implies (\ref{prop:equiv-perm1}).  The reverse implication is even more immediate.  We also note if $G$ has no dominating vertices then (\ref{prop:equiv-perm1}) immediately follows from (\ref{prop:equiv-perm2}).  Therefore, we shall assert $G$ has $m\geq 1$ dominating vertices.  Let $G'\subset G$ denote the subgraph obtained by omitting them.   

Choose a vertex $v\in G$ with degree $d\neq n-1$ and let $V\subseteq G$ denote the order $d$ maximal clique containing it.  The subgraph $V'=V\cap G'$ is also a maximal clique in $G'$, of order $d-m$.  Assume, without loss of generality, that $v$ is adjacent to a vertex $w\in G'\setminus V'$, so that $V'$ does not comprise a connected component of $G'$.  Then since $w$ is not dominating in $G$, nor is $w\in V$, the degree of $v$ must be at least $d+1$, a contradiction.  \end{proof}

\begin{prop} The graph $\Graph(\vect g)$ associated to a point $\vect g\in\mathscr G$ is permissible. \end{prop}

\begin{proof} At least one Pl\"ucker coordinate $g_{\underbar i}$ does not vanish, so $G$ is not a complete graph.  We will show (\ref{prop:equiv-perm2}) from Proposition \ref{prop:equiv-perm} holds for $G$, as well.

Without loss of generality, because of the group actions described in Section \ref{sec:obs}, assert the Pl\"ucker coordinate $g_{\{1,\dots,n-2\}}$ is non-zero.  We can write a matrix  
\begin{center}
\includegraphics[scale=1.05]{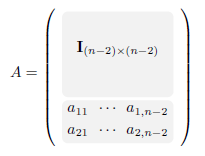}
\end{center}
so that $\col A=\vect g$, and two vertices $v_i,v_j\in G$ are adjacent if and only if the submatrix of $A$ obtained by omitting the $i$th and $j$th rows is singular.  For example, it is obvious that $v_{n-1}$ and $v_n$ are not adjacent.  We also note how every Pl\"ucker coordinate corresponds to either an entry or a 2-minor of the submatrix  
\[A'=\begin{pmatrix}a_{11} & \cdots & a_{1,n-2} \\
	a_{21} & \cdots & a_{2,n-2}
	\end{pmatrix}.\]
We now, again without loss of generality, consider the respective degrees of the vertices $v_1$ and $v_n$. 

First consider $v_i$, where $i\in\{1,\dots,n-2\}$.  If $v_i$ is not adjacent to either of $v_{n-1},v_n$ then in $A'$ some collection of 2-minors, all involving the $i$th column, must vanish.  But the two ways this happens are either $a_{1i}=a_{2i}=0$ and equivalently, $v_i$ is dominating, contradicting the hypothesis $v_i$ is not adjacent to $v_{n-1}$ or $v_n$; or, if all Pl\"ucker coordinates involving entries from columns whose indices give vertices adjacent to $v_i$ vanish, in which case, such vertices, along with $v_i$, form a maximal clique.  

Suppose, other the other hand, $v_i$ is adjacent to, say, $v_{n-1}$.  We already know $v_n$ is not adjacent to $v_{n-1}$, and we have 
\[a_{2i}=\det\,A(\{1,\dots,n\}\setminus\{i,n-1\})=0.\]  
If $v_i$ is also adjacent to $v_n$ then $a_{1i}=0$, in which case the $i$th column of $A'$ vanishes and equivalently, $v_i$ is a dominating vertex.  Suppose then, that $v_i$ is adjacent to $v_{n-1}$ but not to $v_n$.  We claim the vertices adjacent to $v_{n-1}$, along with $v_{n-1}$, must form a clique.  If two such entries $a_{2i},a_{2j}$ vanish, i.e., $v_i,v_j$ are both adjacent to $v_{n-1}$, then so does the 2-minor $a_{1i}a_{2j}-a_{1j}a_{2i}$; this is exactly the condition for $v_i$ and $v_j$ to be adjacent.  Finally, we point out all of the above analysis for $v_{n-1}$ applies analogously to $v_n$.  It follows that all of the vertices in $G$ are either dominating or part of a clique.  \end{proof}

\begin{figure}\centering{
\begin{tikzpicture}[every node/.style={circle,draw,fill=black,inner sep=0pt,minimum width=4pt}]
\node (13v1) {};
\node (13v2) at ($(13v1)+(30:1)$) {};
\node (13v3) at ($(13v1)+(-30:1)$) {};
\draw (13v1)--(13v2);
\node (23v1) at ($(13v1)+(2.5,0)$) {};
\node (23v2) at ($(23v1)+(30:1)$) {};
\node (23v3) at ($(23v1)+(-30:1)$) {};
\draw (23v1)--(23v2)--(23v3);
\node (14v1) at ($(13v1)+(-2,-2)$) {};
\node (14v2) at ($(14v1)+(30:1)$) {};
\node (14v3) at ($(14v1)+(-45:1)$) {};
\node (14v4) at ($(14v1)+(-10:1.5)$) {};
\draw (14v1)--(14v2) (14v3)--(14v4);
\node (24v1) at ($(14v1)+(3,0)$) {};
\node (24v2) at ($(24v1)+(30:1)$) {};
\node (24v3) at ($(24v1)+(-45:1)$) {};
\node (24v4) at ($(24v1)+(-10:1.5)$) {};
\draw (24v1)--(24v2)--(24v4) (24v2)--(24v3);
\node (34v1) at ($(14v1)+(6,0)$) {};
\node (34v2) at ($(34v1)+(30:1)$) {};
\node (34v3) at ($(34v1)+(-45:1)$) {};
\node (34v4) at ($(34v1)+(-10:1.5)$) {};
\draw (34v1)--(34v2)--(34v3)--(34v1);
\node (15v1) at ($(13v1)+(-3.5,-4.5)$) {};
\node (15v2) at ($(15v1)+(50:1)$) {};
\node (15v3) at ($(15v1)+(10:1.5)$) {};
\node (15v4) at ($(15v1)+(-25:1.5)$) {};
\node (15v5) at ($(15v1)+(-65:0.9)$) {};
\draw (15v1)--(15v2) (15v1)--(15v3) (15v1)--(15v4) (15v1)--(15v5);
\node (25v1) at ($(15v1)+(3,0)$) {};
\node (25v2) at ($(25v1)+(50:1)$) {};
\node (25v3) at ($(25v1)+(10:1.5)$) {};
\node (25v4) at ($(25v1)+(-25:1.5)$) {};
\node (25v5) at ($(25v1)+(-65:0.9)$) {};
\draw (25v1)--(25v2) (25v1)--(25v3)--(25v2) (25v1)--(25v4)--(25v2) (25v1)--(25v5)--(25v2);
\node (35v1) at ($(15v1)+(6,0)$) {};
\node (35v2) at ($(35v1)+(50:1)$) {};
\node (35v3) at ($(35v1)+(10:1.5)$) {};
\node (35v4) at ($(35v1)+(-25:1.5)$) {};
\node (35v5) at ($(35v1)+(-65:0.9)$) {};
\draw (35v1)--(35v2)--(35v3)--(35v1) (35v1)--(35v4) (35v1)--(35v5);
\node (45v1) at ($(15v1)+(9,0)$) {};
\node (45v2) at ($(45v1)+(50:1)$) {};
\node (45v3) at ($(45v1)+(10:1.5)$) {};
\node (45v4) at ($(45v1)+(-25:1.5)$) {};
\node (45v5) at ($(45v1)+(-65:0.9)$) {};
\draw (45v1)--(45v2)--(45v3)--(45v4)--(45v1)--(45v3) (45v2)--(45v4);
\end{tikzpicture}
\caption{Permissible graphs of orders 3,4,5.}\label{fig:permissGraphs}
}\end{figure}
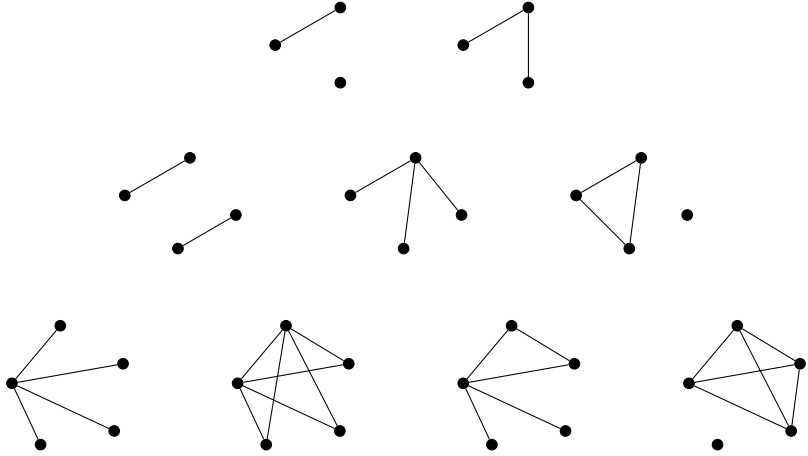

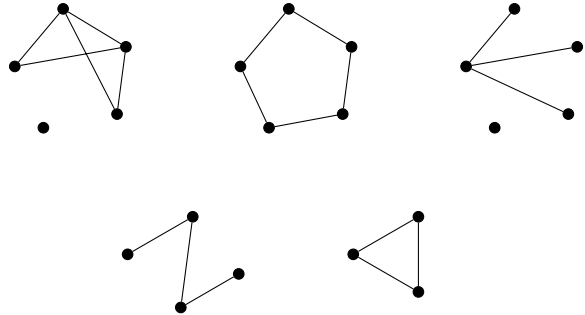
\begin{figure}\centering{
\begin{tikzpicture}[every node/.style={circle,draw,fill=black,inner sep=0pt,minimum width=4pt}]
\node (15v1) {};
\node (15v2) at ($(15v1)+(50:1)$) {};
\node (15v3) at ($(15v1)+(10:1.5)$) {};
\node (15v4) at ($(15v1)+(-25:1.5)$) {};
\node (15v5) at ($(15v1)+(-65:0.9)$) {};
\draw (15v1)--(15v2)--(15v3)--(15v4) (15v1)--(15v3) (15v2)--(15v4);
\node (25v1) at ($(15v1)+(3,0)$) {};
\node (25v2) at ($(25v1)+(50:1)$) {};
\node (25v3) at ($(25v1)+(10:1.5)$) {};
\node (25v4) at ($(25v1)+(-25:1.5)$) {};
\node (25v5) at ($(25v1)+(-65:0.9)$) {};
\draw (25v1)--(25v2)--(25v3)--(25v4)--(25v5)--(25v1);
\node (35v1) at ($(15v1)+(6,0)$) {};
\node (35v2) at ($(35v1)+(50:1)$) {};
\node (35v3) at ($(35v1)+(10:1.5)$) {};
\node (35v4) at ($(35v1)+(-25:1.5)$) {};
\node (35v5) at ($(35v1)+(-65:0.9)$) {};
\draw (35v1)--(35v2) (35v1)--(35v3) (35v1)--(35v4);
\node (14v1) at ($(15v1)+(1.5,-2.5)$) {};
\node (14v2) at ($(14v1)+(30:1)$) {};
\node (14v3) at ($(14v1)+(-45:1)$) {};
\node (14v4) at ($(14v1)+(-10:1.5)$) {};
\draw (14v1)--(14v2)--(14v3)--(14v4);
\node (13v1) at ($(14v1)+(3,0)$) {};
\node (13v2) at ($(13v1)+(30:1)$) {};
\node (13v3) at ($(13v1)+(-30:1)$) {};
\draw (13v1)--(13v2)--(13v3)--(13v1);
\end{tikzpicture}
\caption{These graphs are not permissible.}\label{fig:antiPermiss}
}\end{figure}

It is clear the set of order $n$ permissible graphs injects into $\Grass(n-2,n)$.  Given a permissible graph $G$, the set of points in $\mathscr G$ with that graph is locally closed.  Its closure is all points whose graph contains $G$.

\begin{dfn} A subvariety $\mathscr S\subseteq\mathscr G$ is \emph{permissible} means it is the closure of the set of all points with the same fixed permissible graph, which we denote $\Graph(\mathscr S)$.  \end{dfn}

Recall, given a collection $\underbar I$ of $(n-2)$-element subsets of $\{1,\dots,n\}$, we let $\Var(\underbar I)$ denote the vanishing set of the Pl\"ucker coordinates indexed by the sets in $\underbar I$.  Using a known observation about ideals generated by minors of a $2\times s$ matrix, we show, with $s=n-2$, that the components of $\Var(\underbar I)$ are permissible.  The observation is as follows: if two overlapping 2-minors of a generic $2\times s$ matrix vanish, then either the two entries in the overlap vanish or else the third minor among the three columns in question must vanish. 

\begin{prop}\label{prop:pluckgens} In $\mathscr G=\Grass(n-2,n)$, the irreducible components of $\Var(\underbar I)$, where $\underbar I\subset\{1,\dots,n\}$ has cardinality $n-2$, are permissible. \end{prop}

\begin{proof} Recall, $\mathscr G$ is covered by open affine sets $\mathscr G_{\underbar i}$, where the $\underbar i$th Pl\"ucker coordinate does not vanish.  We shall prove the result in the affine case, i.e., for $\mathscr V=\Var(\underbar I)\cap\mathscr G_{\underbar i}$.  We can choose matrices normalized with respect to $\underbar i$ whose column spaces represent points in $\mathscr V$.  For each class of sets $\underbar I$ not containing $\underbar i$, we construct $G=\Graph(\col A)$, having supposed $\col A\in\Var(\underbar I)$.  

We claim the components of $\mathscr V$ are in bijection with the minimal possible graphs associated to its points, in the sense that removing any edge from such a graph gives a graph to which no point in $\mathscr V$ is associated.  Let $I$ denote the homogeneous defining ideal for $\mathscr V$, generated by the Pl\"ucker variables with indices in $\underbar I$.    

Write $\underbar i=\{1,\dots,n\}\setminus\{i,j\}$ and choose $\vect g\in\mathscr V$.  The Pl\"ucker coordinates for $\vect g$ are, up to a scalar, the minors of the submatrix $A\left(\{i,j\};\{1,\dots,n-2\}\right)$.  Let
\[\begin{split}
A' &=A\left(\{i,j\};\{1,\dots,n-2\}\right) \\
&=\begin{pmatrix}a_{i1} & \cdots & a_{i,n-2} \\
	a_{j1} & \cdots & a_{j,n-2}\end{pmatrix}.
\end{split}\]
We now look for instances where such a matrix $A$, i.e., with $\col A=\vect g$, has a minimal graph.  If two algebraically independent Pl\"ucker coordinates vanish then $\Graph(A)$ does not require any more edges than the two given.  We must consider the cases where non-algebraically independent Pl\"ucker coordinates vanish, since the Pl\"ucker relations may require the vanishing of additional Pl\"ucker coordinates. 

Suppose two overlapping 2-minors of $A'$ vanish.  This happens if  either the third 2-minor in the three involved columns vanishes, or the two entries in the overlapping column vanish.  This first case corresponds to a triangle (3-cycle) in $\Graph(A)$.  In the other case the vertex whose index matches the vanishing column is dominating.  On the other hand, two non-overlapping 2-minors of $A'$ are algebraically independent of each other, so the homogeneous ideal generated by their respective Pl\"ucker coordinates is prime.  

The next case we consider is when some collection of entries of $A'$ vanish.  If two entries in the same column vanish then all 2-minors involving that column vanish and the corresponding vertex in $\Graph(A)$ is dominating.  If two entries in the same row vanish then the 2-minor involving those entries vanishes and this is reflected in $\Graph(A)$ as a triangle.  Note, the entries of $A'$ themselves are algebraically independent of each other.  If there are no other generators for $I$, then $\Var(\underbar I)$ is a permissible subvariety.   

The final case to consider is when we suppose an entry $a$ of $A'$ and a 2-minor $\mu$ of $A'$, containing $a$, vanish.  Write
\[\mu=\left|\;\begin{matrix}a & b \\[-0.5pc]
	c & d \end{matrix}\;\right|.
\]
We lose no generality, as the position of $a$ only determines a name change of the other entries in order to get the same formula, $ad-bc$, for $\mu$.  It follows either $b$ or $c$ must also vanish.  If $c$ vanishes, then all 2-minors involving the column $\left(\begin{smallmatrix}a \\ c\end{smallmatrix}\right)$ vanish, and the vertex with the same index as that column is dominating.  If $b=0$ then in $\Graph(A)$ we get a triangle.  As we can see, vanishing of any collection of Pl\"ucker coordinates can only cause other Pl\"ucker coordinates to vanish.  We conclude the minimal primes for $\Var(\underbar I)$ are generated by Pl\"ucker coordinates, and thus, each have a unique corresponding graph.
\end{proof}

\subsection{Consequences}\label{sec:conseq}

It is not hard to see any permissible graph $G$ will contain either isolated vertices, dominating vertices, or neither, but not both. Let $G_{triv}$ denote this set of vertices.  We will say $G_{triv}$ is \emph{dominating} to mean its vertices are dominating, or \emph{isolated}, to mean its vertices are isolated.  $G_{triv}$ may be empty, and $\left|G_{triv}\right|\neq n$.  By permissibility of $G$, the set $G\setminus G_{triv}$ is a disjoint union of $c$ cliques of respective orders $a_1,\dots,a_c$.  

\begin{thm}\label{thm:codimCalc} Suppose $\mathscr S\subset\mathscr G=Grass(n-2,n)$ is a permissible subvariety with graph $G=\Graph(\mathscr S)$.  Let $a_1,\dots,a_c$ denote the respective orders of the cliques in $G\setminus G_{triv}$.  Put $m=\left|G_{triv}\right|$ and $l=\sum_{j=1}^c(a_j-1)$.  Then the codimension of $\mathscr S$ in $\mathscr G$ is 
\[\codim\,\mathscr S=\begin{cases}n-c+m = 2m + l & \text{if $G_{triv}\neq\emptyset$, dominating} \\
	n-c-m = l & \text{otherwise.}\end{cases}\]
\end{thm}

\begin{proof} An isolated vertex contributes nothing to the codimension.  Say $A$ is an $n\times(n-2)$ matrix, normalized with respect to some set of indices $\underbar i=\{1,\dots,n\}\setminus\{i,j\}$, and such that  $\Graph(A)=G$.  Let 
\[A'=\begin{pmatrix}a_{i1} & \cdots & a_{i,n-2} \\
	a_{j1} & \cdots & a_{j,n-2}\end{pmatrix},\]
the complementary submatrix to the order $n-2$ identity submatrix of $A$.  A dominating vertex $v\in G$ indicates that two entries on a column of $A'$ vanish, contributing 2 to the codimension.  All other edges joined to $v$ correspond to 2-minors of $A'$ involving that vanishing column and thus contribute nothing to the codimension.

Put $G'=G\setminus G_{triv}$.  Columns of $A'$ involved in a given clique of $G'$ are independent of those involved in minors corresponding to edges joined to dominating vertices.  Finally, there are two cases left to consider.  The first case is where a clique of order $a\geq 2$ indicates a collection of $a-1$ entries from the same row of $A'$ vanish, so contributes $a-1$ to the codimension.  The other case is when a clique of order $a\leq 2$ indicates all 2-minors involving some set of $a$ columns in $A'$ vanish.  The condition is equivalent to the condition that a generic $2\times a$ matrix have rank 1.  It is known (see \cite{hochster+eagon}) that a generic rank 1 matrix of size $2\times a$ defines an ideal of height $a-1$.  \end{proof}

Given a permissible graph $G$, let $H$ denote its complement.  In understanding components of $\Theta(\mathscr Y_{n,n-2,n-2})$ in Equation (\ref{eq:Theta}), we wish to minimally enlarge $H$ to a permissible graph, $\tilde{H}$, and then take a minimal permissible subgraph $\tilde G\subseteq G$ such that together, $\tilde G,\tilde H$ cover $K_n$.  Specifically, $\tilde H$ should not properly contain any permissible subgraph containing $H$, and $\tilde G$ should not contain any permissible subgraph that, with $\tilde H$, forms a covering.  Upon finding such a pair $(\tilde G,\tilde H)$, we let $(\mathscr S,\mathscr T)$ denote the pair of permissible subvarieties with the respective graphs.  

\begin{thm}\label{thm:permPairs} A minimal pair, up to permutation, of permissible subvarieties $(\mathscr S,\mathscr T)$ whose associated graphs form a covering must satisfy: 
\begin{enumerate}[(a)]
\item $\Graph(\mathscr S)$ consists of a clique of order $a$ with the remaining vertices isolated, and 
\item $\Graph(\mathscr T)$ is its complement, an order $n$ graph with $n-a$ dominating vertices.  
\end{enumerate}
These pairs completely describe the components of $\mathscr Y_{n,n-2,n-2}$.  The number of cliques in $\Graph(\mathscr S)$ may take on any value $2\leq a\leq n-1$. \end{thm}

\begin{proof} We now give an algorithm for producing such a pair from a fixed permissible graph $G$.  Let $G_{triv}$ denote the (possibly empty) set of isolated or dominating vertices of $G$, let $G'=G\setminus G_{triv}$ and let $H'$ denote the complement of $G'$.  Let $a_1,\dots,a_c$ denote the sizes of the respective cliques in $G'$.

Suppose $G_{triv}$ is non-empty and consists of dominating vertices.  If $H'$ is permissible and $a_j=1$ for all $j=1,\dots,c$ then let $H$ denote the union of $H'$ and the vertices from $G_{triv}$, so that $G,H$ give respective graphs for permissible subvarieties $\mathscr S,\mathscr T$ and we are done.  If, on the other hand, $H'$ is permissible but $a_j>1$ for some $j$ then $a_i=1$ for all $i\neq j$ and we have two ways to enlarge $H'$: 
\begin{compactenum}[\hspace{7.5pt}A)]
\item Complete $H'$, then let $\tilde H$ denote its union with the vertices from $G_{triv}$.  Then let $\tilde G\subseteq G$ denote the complement of $\tilde H$.  $\tilde G$ is permissible because it consists of the edges incident to vertices in $G_{triv}$.  
\item There is at least one isolated vertices in $G'$.  To construct $\tilde H$ make the isolated vertices from $G'$ into dominating vertices.  Remove their edges from $G$ to get a subgraph $\tilde G$.   
\end{compactenum}
If $H'$ is not permissible, and $G_{triv}$ is non-empty and consists of dominating vertices, then we can either do A), as above, or we can do the following: choose a clique $B$ from $G'$ of order $a_j\geq 2$.  In constructing $\tilde H$, make all vertices in $G\setminus B$ dominating.  The complement, $\tilde G$, of $\tilde H$ is a clique of order $a_j+m$, with the remaining vertices isolated.

To finish the proof, now suppose $G_{triv}$ is either empty or consists of isolated vertices.  If $H'$ is permissible then adding the vertices from $G_{triv}$ to $H'$ and making them dominating does not change permissibility and we are done.  If $H'$ is not permissible then let $H$ denote $H'$, together with the vertices from $G_{triv}$ as dominating vertices.  Choose $j$ such that $a_j>1$ and enlarge $H$ by making all vertices not in that clique, call it $B$, dominating.  The complement $\tilde G$ is exactly the clique $B$.
\end{proof}

\begin{thm}\label{thm:n-2,n-2} The locally closed set $\mathscr Y_{n,n-2,n-2}$, of $n\times n$ matrices of rank $n-2$ whose size $n-2$ principal minors vanish, has dimension $n^2-4-n$. \end{thm}

\begin{proof} A matrix $A\in\mathscr Y_{n,n-2,n-2}$ has a normalized factorization given by $2(n-2)+(n-2)^2+2(n-2)=n^2-4$ parameters.  Subtract the minimal codimension of $\mathscr S\times\mathscr T$, as computed in Theorem \ref{thm:codimCalc}, over all possible pairs as described in Theorem \ref{thm:permPairs}. \end{proof}

In Section \ref{sec:n=5} we illustrate these results for the case $n=5$.  In the meantime, however, we digress briefly to discuss an attempt to extend the results to $t=n-3$.

\subsection{Components of $\mathscr Y_{n,n-3,n-3}$}\label{sec:n-3}

As in the $r=t=n-2$ case a matrix $A\in\mathscr Y_{n,n-3,n-3}$ factors so that we may identify it with a pair of points in $\Grass(n-3,n)$.  In the spirit of Section \ref{sec:pluckersToGraphs}, to any set of Pl\"ucker coordinates we may associate a simplicial complex that is a union of 2-simplices.  There is a notion of permissibility; permissible 2-complexes are the ones that actually come from a matrix.  We can then ask the following:

\begin{que} Given a permissible 2-complex, is the closure of the algebraic set defined by it irreducible? \end{que}

\begin{que} Is every algebraic set defined by vanishing of Pl\"ucker coordinates a union of 2-permissible ones which are irreducible? \end{que}

The problem reduces to finding the conditions for a set of minors of a generic $3\times(n-3)$ matrix to define a prime ideal.   

\begin{ex} We used {\it Macaulay2} (\emph{M2}) for $n=8$ and $K=\Z/101\Z$ to compute the minimal primes for various collections of Pl\"ucker variables.  Letting
\[U=\begin{pmatrix}1 & 0 & 0 & 0 & 0 \\
0 & 1 & 0 & 0 & 0 \\
0 & 0 & 1 & 0 & 0 \\
0 & 0 & 0 & 1 & 0 \\
0 & 0 & 0 & 0 & 1 \\
u_{61} & u_{62} & u_{63} & u_{64} & u_{65} \\
u_{71} & u_{72} & u_{73} & u_{74} & u_{75} \\
u_{81} & u_{82} & u_{83} & u_{84} & u_{85} \end{pmatrix}
\]
parametrize a matrix whose column space is a point in $\Grass(5,8)$, we considered ideals in $\Z/101\Z[\wedge^5U]$.  

In Figure \ref{fig:overlapping2} we associate a 2-simplex to the ideal  
\[(u_{61}u_{72}-u_{62}u_{71},\,u_{62}u_{73}-u_{63}u_{72}).\]
For simplicity, the vertices are labelled only by their numerical subscripts and axes are drawn to give perspective to the configuration.  According to \emph{M2} the minimal primes are
\[P_1=(u_{61}u_{72}-u_{62}u_{71},\,u_{62}u_{73}-u_{63}u_{72},\,u_{61}u_{73}-u_{63}u_{71})\]
and
\[P_2=(u_{62},u_{72}),\]
both of which are also given by Pl\"ucker variables.  Their corresponding 2-simplices give a picture of what a permissible 2-simplex ``ought" to look like.  Different colors are used for different sized minors.

\begin{figure}
\includegraphics[page=4,width=1\textwidth]{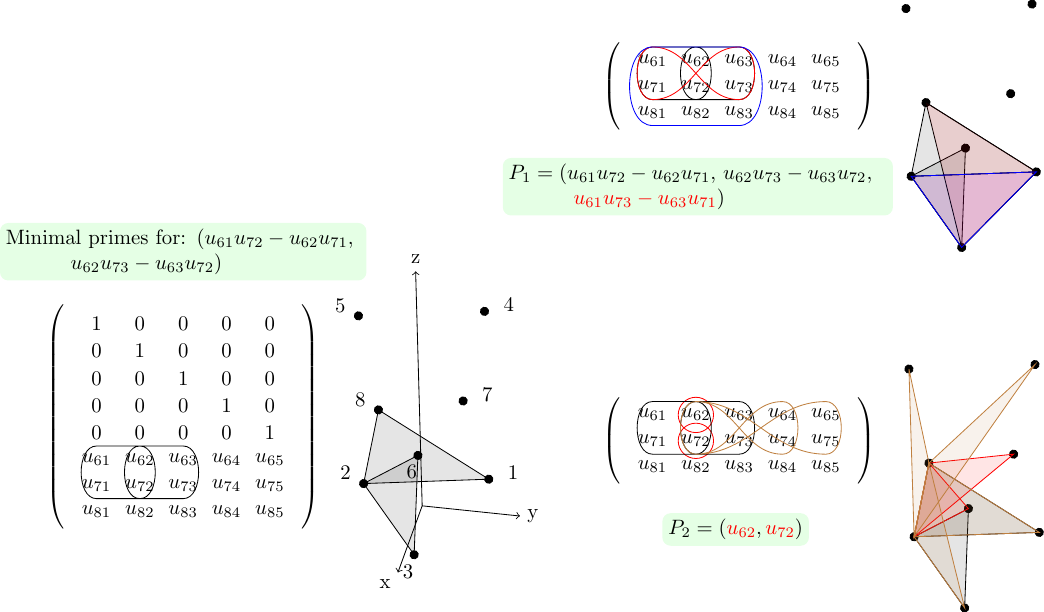}
\caption{Minimal primes for an ideal generated by two overlapping 2-minors of a $3\times(n-3)$ matrix, for $n=8$, along with their respective associated 2-simplices.}\label{fig:overlapping2}
\end{figure}

Similarly, Figure \ref{fig:overlapping3} shows the ideal 
\[(u_{61}u_{72}-u_{62}u_{71},\,u_{62}u_{73}-u_{63}u_{72},\,u_{72}u_{83}-u_{73}u_{82})\]
and its minimal primes $Q_1,Q_2,Q_3$, as computed by \emph{M2}.

\begin{figure}
\includegraphics[page=5,width=1\textwidth]{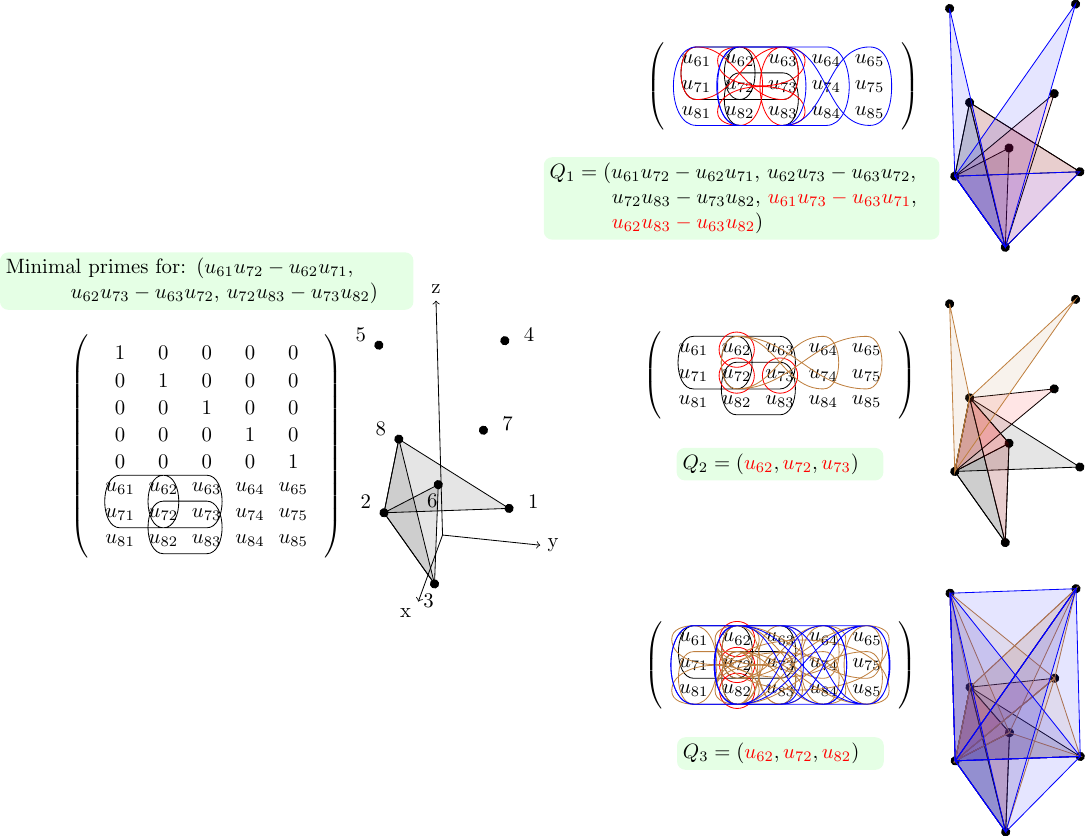}
\caption{Minimal primes for an ideal generated by three overlapping 2-minors of a $3\times(n-3)$ matrix, for $n=8$, along with their respective associated 2-simplices.}\label{fig:overlapping3}
\end{figure}
\end{ex}

In Proposition \ref{prop:case3} we show that an ideal generated by a $3\times 3$ minor and one of its nested $2\times 2$ minors has two minimal primes, also generated by minors.  

\begin{ex} Suppose $X=(x_{ij})$ is a $3\times 3$ generic matrix and let 
\[I=(\det\,X,\,x_{11}x_{22}-x_{12}x_{21})\subset K[X].\]
By Proposition \ref{prop:case3}, the minimal primes for $I$ are
\begin{align*}
(x_{11}x_{22}-x_{12}x_{21},\,x_{11}x_{23}-x_{13}x_{21},\,x_{12}x_{23}-x_{13}x_{22}) & \quad\text{ and } \\
(x_{11}x_{22}-x_{12}x_{21},\,x_{11}x_{32}-x_{12}x_{31},\,x_{21}x_{32}-x_{22}x_{31})
&,
\end{align*}
given by the highlighted minors:
\[\begin{tikzpicture}
\matrix (P1) [matrix of math nodes,left delimiter=(,right delimiter=)] {x_{11} & x_{12} & x_{13} \\
	x_{21} & x_{22} & x_{23} \\
	x_{31} & x_{32} & x_{33} \\};
\matrix (P2) [matrix of math nodes,left delimiter=(,right delimiter=)] at ($(P1)+(4,0)$) {x_{11} & x_{12} & x_{13} \\
	x_{21} & x_{22} & x_{23} \\
	x_{31} & x_{32} & x_{33} \\};
\draw[red,rounded corners] 
	(P1-1-1.north) to[out=0,in=180] (P1-1-2.north) 
	(P1-1-2.north) to[out=0,in=0] (P1-2-2.south) 
	(P1-2-2.south) to[out=180,in=0] (P1-2-1.south) 
	(P1-2-1.south) to[out=180,in=180] (P1-1-1.north) 
	(P1-1-1.north) to[out=0,in=180] (P1-2-3.south)
	(P1-2-3.south) to[out=0,in=0] (P1-1-3.north)
	(P1-1-3.north) to[out=180,in=0] (P1-2-1.south)
	(P1-2-1.south) to[out=180,in=180] (P1-1-1.north) 
	(P1-1-2.north) to[out=0,in=180] (P1-1-3.north) 
	(P1-1-3.north) to[out=0,in=0] (P1-2-3.south) 
	(P1-2-3.south) to[out=180,in=0] (P1-2-2.south) 
	(P1-2-2.south) to[out=180,in=180] (P1-1-2.north) 
	(P2-1-1.north) to[out=0,in=180] (P2-1-2.north) 
	(P2-1-2.north) to[out=0,in=0] (P2-2-2.south) 
	(P2-2-2.south) to[out=180,in=0] (P2-2-1.south) 
	(P2-2-1.south) to[out=180,in=180] (P2-1-1.north) 
	(P2-1-1.north) to[out=0,in=180] (P2-3-2.south)
	(P2-3-2.south) to[out=0,in=-90] (P2-3-2.east)
	(P2-3-2.east) to[out=90,in=90] (P2-3-1.west)
	(P2-3-1.west) to[out=-90,in=180] (P2-3-1.south)
	(P2-3-1.south) to[out=0,in=180] (P2-1-2.north)
	(P2-1-2.north) to[out=0,in=90] (P2-1-2.east)
	(P2-1-2.east) to[out=-90,in=-90] (P2-1-1.west)
	(P2-1-1.west) to[out=90,in=180] (P2-1-1.north) 
	(P2-2-1.north) to[out=0,in=180] (P2-2-2.north) 
	(P2-2-2.north) to[out=0,in=0] (P2-3-2.south) 
	(P2-3-2.south) to[out=180,in=0] (P2-3-1.south) 
	(P2-3-1.south) to[out=180,in=180] (P2-2-1.north) 
	;
\end{tikzpicture}
\]
\end{ex}  

\begin{prop}\label{prop:case3} For a $3\times 3$ generic matrix $X$ the ideal in $K[X]$ generated by $\det\,X$ and the 2-minor $x_{i_1j_1}x_{i_2j_2}-x_{i_1j_2}x_{i_2j_1}$ has two minimal primes:
\begin{align*}
(x_{i_1j_1}x_{i_2j_2}-x_{i_1j_2}x_{i_2j_1},\,x_{i_1j_1}x_{i_2j_3}-x_{i_1j_3}x_{i_2j_1},\,x_{i_1j_2}x_{i_2j_3}-x_{i_1j_3}x_{i_2j_2})\; & \quad\text{ and} \\
(x_{i_1j_1}x_{i_2j_2}-x_{i_1j_2}x_{i_2j_1},\,x_{i_1j_1}x_{i_3j_2}-x_{i_1j_2}x_{i_3j_1},\,x_{i_2j_1}x_{i_3j_2}-x_{i_2j_2}x_{i_3j_1}). &
\end{align*}
\end{prop}

\begin{proof} Write $X=\left(\begin{smallmatrix}x_{11} & x_{12} & x_{13} \\
	x_{21} & x_{22} & x_{23} \\
	x_{31} & x_{32} & x_{33} \end{smallmatrix}\right)$.
For the proof we simplify the notation; let
\[\begin{split}
\mu &= x_{i_1j_1}x_{i_2j_2}-x_{i_1j_2}x_{i_2j_1} \\
\Delta & =\det\,X \\
I &= (\Delta,\mu)\subset K[X] \\
R &= K[X]/I \\
\Delta_{rs}^{tu} &= x_{rt}x_{su}-x_{ru}x_{st},\,\text{ for }r,s,t,u\in\{1,2,3\} \\
\left(\text{so that }\mu\right. &= \Delta_{i_1i_2}^{j_1j_2}\left.\right). 
\end{split}\]
Translating to the new notation, we wish to show the minimal primes for $I$ are 
\[(\mu,\Delta_{i_1i_2}^{j_1j_3},\Delta_{i_1i_2}^{j_2j_3})\quad\text{ and }\quad(\mu,\Delta_{i_1i_3}^{j_1j_2}\Delta_{i_2i_3}^{j_1j_2}).\]

We shall localize at $x=x_{i_3j_2}$.  We first show $x$  is not a zerodivisor on $R$ by showing $J=I+(x)$ has height 3 in $K[X]$.  (The ideal $I$ has two generators, so its height is at most 2.)  In $K[X]/J$, 
\[\begin{split}
0 &= \Delta \\
&= \pm\; x_{i_3j_1}\Delta_{i_1i_2}^{j_2j_3}\;\mp\; x\Delta_{i_1i_2}^{j_1j_3}\;\pm\; x_{i_3j_3}\mu \\
& \text{(note, the signs are actually ambiguous, but in this case, are also irrelevant)} \\
&= \pm\;x_{i_3j_1}\Delta_{i_1i_2}^{j_2j_3}
\end{split}\]	
implies we can decompose 
\[\begin{split}
\Var(J) &= \Var(\mu,x,x_{i_3j_1})\;\cup\;\Var(\mu,x,\Delta_{i_1i_2}^{j_2j_3}) \\
&= \Var\underbrace{(\mu,x,x_{i_3j_1})}_{\text{height $=3$}}
	\;\cup\;
	\Var\underbrace{(\mu,\Delta_{i_1i_2}^{j_1j_3},\Delta_{i_1i_2}^{j_2j_3},x)}_{\text{height $=2+1=3$}}
	\;\cup\;\Var\underbrace{(x_{i_1j_2},x_{i_2j_2},x)}_{\text{height $=3$}}, \\
& \text{ since }(\mu,\Delta_{i_1i_2}^{j_2j_3})=(\mu,\Delta_{i_1i_2}^{j_1j_3},\Delta_{i_1i_2}^{j_2j_3})\cap(x_{i_1j_2},x_{i_2j_2}).	
\end{split}\]
Thus $J$ has height 3 and it follows we can invert the element $x$.  

Over the localized ring $R_x=R[\frac{1}{x}]$ we can clear the remaining entries in row $i_3$ of $X$; let $X'=(x_{ij}')$ denote the resulting matrix.  We shall index its minors the same way we did for $X$, only replacing $\Delta$ with $\delta$, e.g., we let $\delta_{rs}^{tu}$ denote the minor $x_{rt}'x_{su}'-x_{ru}'x_{st}'$ for $r,s,t,u\in\{1,2,3\}$.  The entries for $X'$ are \[\begin{array}{r@{\,=\,}l@{\;\qquad\;}r@{\,=\,}l@{\;\qquad\;}r@{\,=\,}l}
	x'_{i_1j_1} & x_{i_1j_1}-\frac{x_{i_3j_1}}{x}x_{i_1j_2} & 
	x'_{i_1j_2} & x_{i_1j_2}
	 &
	x'_{i_1j_3} & x_{i_1j_3}-\frac{x_{i_3j_3}}{x}x_{i_1j_2} 
	\\
	x'_{i_2j_1} & x_{i_2j_1}-\frac{x_{i_3j_1}}{x}x_{i_2j_2}
	 &
	x'_{i_2j_2} & x_{i_2j_2}
	 &
	x'_{i_2j_3} & x_{i_2j_3}-\frac{x_{i_3j_3}}{x}x_{i_2j_2}\\
	x'_{i_3j_1} & 0
	 &
	x'_{i_3j_2} & 1 
	 &
	x'_{i_3j_3} & 0
\end{array}.\]
Expanding along the $i_3$th row (again, the signs are ambiguous, but irrelevant),
\[\Delta=\delta=\pm\;0\cdot\delta_{i_1i_2}^{j_2j_3}\;\mp\;1\cdot\delta_{i_1i_2}^{j_1j_3}\;\pm\;0\cdot\delta_{i_1i_2}^{j_1j_2},\]
and $\mu=\delta_{i_1i_2}^{j_1j_2}$.  Thus $IR_x$ has the decomposition
\[IR_x=(\delta_{i_1i_2}^{j_1j_2},\delta_{i_1i_2}^{j_1j_3})R_x=(\delta_{i_1i_2}^{j_1j_2},\delta_{i_1i_2}^{j_1j_3},\delta_{i_1i_2}^{j_2j_3})R_x\cap(x'_{i_1j_1},x'_{i_2j_2})R_x.\]
The respective contractions to $R$ are the minimal primes for $I$.  For the first ideal,
\[\begin{split}
R\cap(\delta_{i_1i_2}^{j_1j_2},\delta_{i_1i_2}^{j_1j_3},\delta_{i_1i_2}^{j_2j_3})R_x 
&= \left(\mu,\,\Delta_{i_1i_2}^{j_1j_3}-\frac{x_{i_3j_1}}{x}\Delta_{i_1i_2}^{j_2j_3}-\frac{x_{i_3j_3}}{x}\mu,\,\Delta_{i_1i_2}^{j_2j_3}\right):x^{\infty} \\
&= \left(\mu,\,\Delta_{i_1i_2}^{j_1j_3},\,\Delta_{i_1i_2}^{j_2j_3}\right):x^{\infty}  \\
&= \left(\mu,\,\Delta_{i_1i_2}^{j_1j_3},\,\Delta_{i_1i_2}^{j_2j_3}\right),
\end{split}\]
since a prime ideal is already saturated.  For the other prime, since, by hypothesis, $\mu\in(x'_{i_1j_1},x'_{i_2j_1})R_x$, we can write 
\[\begin{split}
R\cap(x'_{i_1j_1},x'_{i_2j_1})R_x &= \left(\mu,\,\frac{1}{x}\Delta_{i_1i_3}^{j_1j_2},\,\frac{1}{x}\Delta_{i_2i_3}^{j_1j_2}\right):x^{\infty} \\
&= \left(\mu,\,\Delta_{i_1i_3}^{j_1j_2},\,\Delta_{i_2i_3}^{j_1j_2}\right),
\end{split}\]
as desired.
\end{proof}

\subsubsection{Generalizing Further} For $r=t<n-3$ or, equivalently, $r=t>3$, the combinatorial approach described above becomes too difficult.  However, these cases always reduce to studying pairs of closed sets in a Grassmannian.  Ford (\cite{ford}) describes a particular type of projective subvariety of the Grassmannian called a \emph{matroid variety}.  Matroids are a type of combinatorial data used to describe many seemingly unrelated objects in mathematics, including graphs, transversals, vector spaces, and networks.  See \cite{pitsoulis} for a recent survey.  Ford computes the codimension of a specific matroid variety called a \emph{positroid variety}.  Positroid varieties, are particularly special; Knutson, Lam and Speyer (\cite{knutson+lam+speyer}) show positroid varieties not only have defining ideals given exactly by Pl\"ucker variables, but such varieties are Cohen-Macaulay, normal, and have rational singularities.  Thus it is worthwhile to ask, are there any conditions where a component of $\mathscr Y_{n,t,t}$ is a product of positroid varieities.
 
\section{Explicit Case: Theorem \ref{thm:n-2,n-2} for $n=5$}\label{sec:n=5}

We explain Theorems \ref{thm:codimCalc}-\ref{thm:n-2,n-2} by focusing on the first non-trivial case, $n=5$.  A matrix $A\in\mathscr Y_{5,3,3}$ has rank $3$ and its size $3$ principal minors vanish.  We have the identification:
\[\begin{split}
\mathscr Y_{5,3,3}\to & \Grass(3,5)\times\Grass(3,5) \\
A\mapsto & \left(\col\,A,\row\,A\right)
\end{split}\]
Without loss of generality, say $\underbar i=\{1,2,3\}$ and $\underbar j=\{1,2,4\}$ index the respective Pl\"ucker coordinates of $(\col\,A,\row\,A)$ which do not vanish.  The factorization
\begin{equation}\label{eq:Afactored}A=\begin{pmatrix}1 & 0 & 0 & \\[-0.25pc]
	0 & 1 & 0 \\[-0.25pc]
	0 & 0 & 1 \\[-0.25pc]
	b_{41} & b_{42} & b_{43} \\[-0.25pc]
	b_{51} & b_{52} & b_{53} \end{pmatrix}
	\cdot A\left(\{1,2,3\};\{1,2,4\}\right)\cdot
	\begin{pmatrix}1 & 0 & c_{13} & 0 & c_{15} \\[-0.25pc]
		0 & 1 & c_{23} & 0 & c_{25} \\[-0.25pc]
		0 & 0 & c_{33} & 1 & c_{35} \end{pmatrix}
\end{equation}
shows $(2\times 3)+(3\times 3)+(2\times 3)=21=25-4$ parameters, not yet considering the requirement that the size 3 principal minors of $A$ vanish.  Now, the principal 3-minors of $A$ vanish if and only if the diagonal entries of $\wedge^3A$ vanish, if and only if for each $i=1,\dots,10$, the $i$th entry of either the column vector $\wedge^3B$ or the row vector $\wedge^3C$ vanishes.  

\begin{figure}
\centering{\includegraphics[page=2,width=0.85\textwidth]{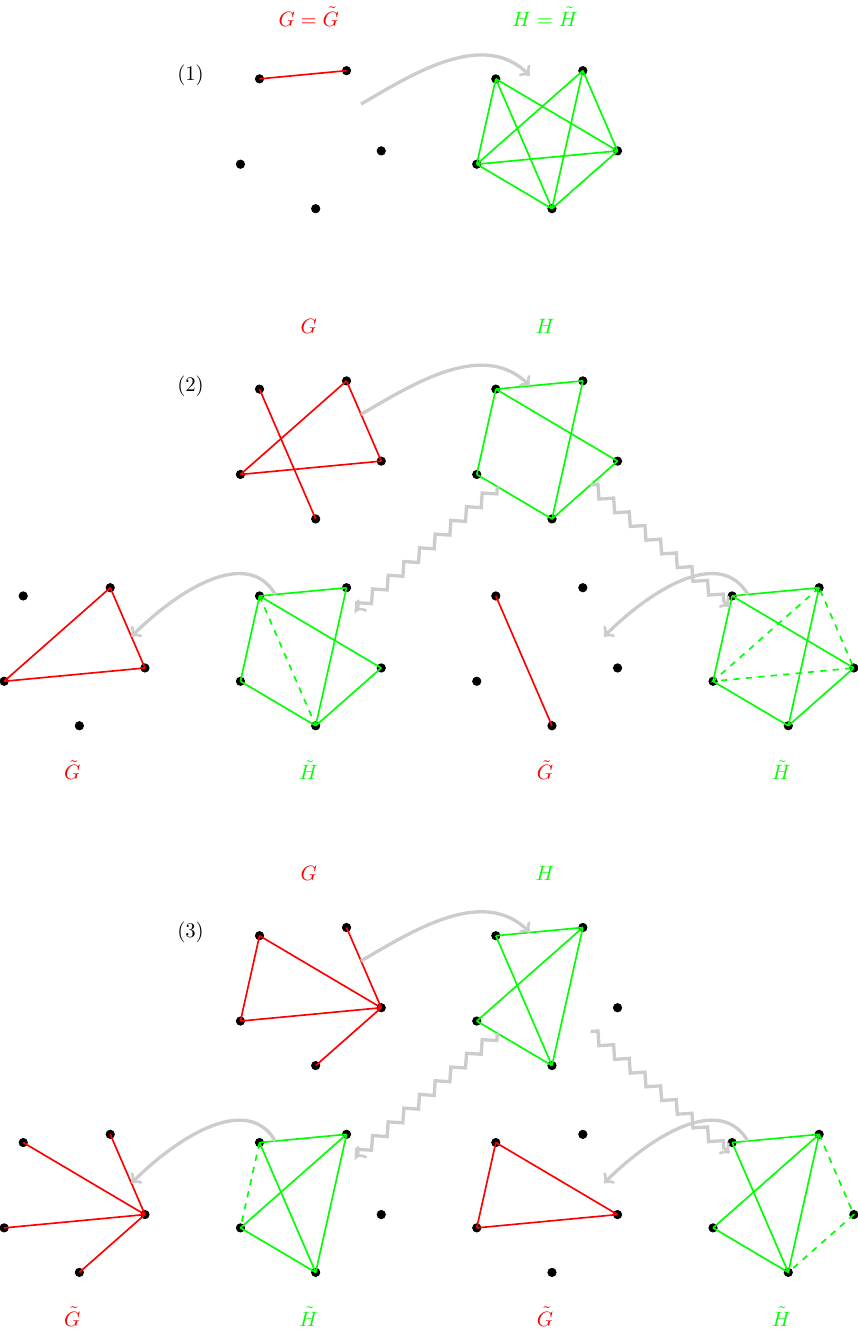}}
\caption{What are the minimal pairs of permissible graphs that cover $n=5$ vertices?  We begin with a permissible red graph, $G$. The green graph to its right is its complement, $H$.  The arrows point to minimal ways to enlarge $H$ to make it permissible; in (1), $H$ is already permissible.  After enlarging $H$ to $\tilde H$, we remove as many edges from $G$ to obtain $\tilde G$, such that $\tilde G,\tilde H$ still form a covering.  It turns out $\tilde G$ will always be permissible, and furthermore, will always be the complement of $\tilde H$.}\label{fig:permiss} 
\end{figure}

\begin{figure}
\centering{\includegraphics[page=1, width=0.5\textwidth]{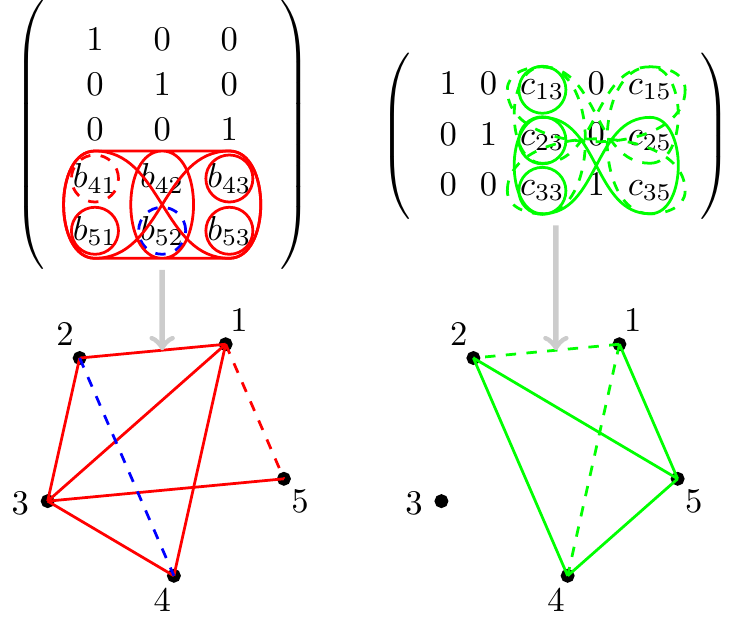}}
\caption{The vertices are labelled only with indices to simplify notation.  In the graphs, an edge joining vertices $v$ and $v'$ is drawn if and only if the Pl\"ucker coordinate with index $\{1,\dots,5\}\setminus\{v,v'\}$ vanishes.  The dotted lines indicate Pl\"ucker coordinates which vanish as a consequence of the solid ones vanishing.  On the left, either the red or the blue dotted line is necessary.}\label{fig:encoding}
\end{figure}

\begin{figure}
\centering{\includegraphics[page=3,width=0.5\textwidth]{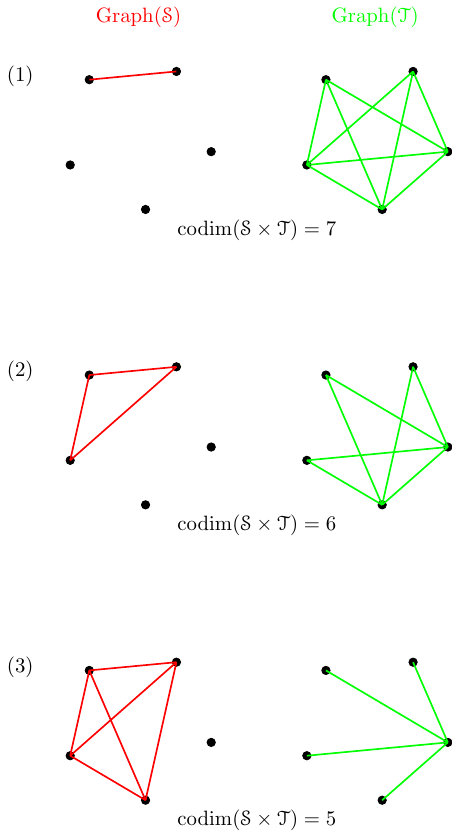}}
\caption{Characterization of the permissible subvarieties $\mathscr S\times \mathscr T\subset\Grass(3,5)\times\Grass(3,5)$ that give components of $\mathscr Y_{5,3,3}$.}\label{fig:thm-permiss}
\end{figure}

\begin{ex} We give a quick example of a possible point $A\in\mathscr Y_{5,3,3}$.  Put $A$ as in (\ref{eq:Afactored}), where we set the colored expressions from Equation (\ref{eq:systemEx}) equal to 0.  

\begin{equation}\label{eq:systemEx}
\left.\begin{array}{r}
	1({\color{green}c_{33}}) \\ 
	{\color{red}b_{43}}(1) \\
	{\color{red}b_{53}}c_{35} \\
	-b_{42}{\color{green}c_{23}} \\
	-b_{52}({\color{green}c_{23}c_{35}-c_{25}c_{33}}) \\
	({\color{red}b_{42}b_{53}-b_{43}b_{52}})(-c_{25}) \\
	b_{41}({\color{green}-c_{13}}) \\
	{\color{red}b_{51}}(-c_{13}c_{35}+c_{15}c_{33}) \\
	({\color{red}-b_{41}b_{53}+b_{43}b_{51}})c_{15} \\ 
	({\color{red}-b_{41}b_{52}+b_{42}b_{51}})(c_{13}c_{25}-c_{15}c_{23}) 
	\end{array}\right\}=0.
\end{equation}
The Pl\"ucker indices for the chosen expressions comprise $\underbar I,\underbar J$:
\[\begin{split}
\underbar I&=\left\{\{1,2,4\},\{1,2,5\},\{1,4,5\},\{2,3,5\},\{2,4,5\},\{3,4,5\}\right\} \\
\underbar J&=\left\{\{1,2,3\},\{1,3,4\},\{1,3,5\},\{2,3,4\}\right\}
\end{split}\]
The solution is shown as the circled Pl\"ucker coordinates in Figure \ref{fig:encoding}.  Notice how the highlighted solution in (\ref{eq:systemEx}) implies the vanishing of additional Pl\"ucker coordinates, as described in the proof of Proposition \ref{prop:pluckgens}; in particular, we have
\[\left.\begin{array}{r}-b_{41}b_{52}+b_{42}b_{51} \\
	b_{51}\end{array}\right\}=0 
\]
implies either $b_{41}=0$ or $b_{52}=0$ and	
\[\left.\begin{array}{r}c_{33} \\
	c_{23} \\
	-c_{13}\end{array}\right\}=0 
\]
implies both $-c_{13}c_{35}+c_{15}c_{23}=0$ and $c_{13}c_{25}-c_{15}c_{23}=0$.
In Figure \ref{fig:encoding} the dotted lines indicate other Pl\"ucker coordinates which vanish as a consequence, making the respective graphs for $\col\,A,\row\,A$ permissible.  The different colored dotted lines in the lefthand matrix and graph reflect the condition that only one of $b_{41}$ or $b_{52}$ is required to vanish. 
\end{ex}

For any $A\in\mathscr Y_{5,3,3}$, we wish to find minimal pairs of permissible subvarieties whose respective graphs cover $K_5$.  Figure \ref{fig:permiss} shows examples of how to construct a pair of permissible graphs which cover $K_5$, given an arbitrary partition of the Pl\"ucker coordinates, i.e., a covering of $K_5$ using a permissible graph and its complement.  Figure \ref{fig:thm-permiss} shows the types of configurations that give a minimally permissible pair of subvarieties.

\section{Conclusion}

Describing the minimal primes for $\pr_{n-2}$ remains incomplete until we have analyzed the components of $\mathscr Y_{n,n-1,n-2}$.  A complete description of $\pr_{n-2}$ would be useful particularly for $n=5$, because we would have a complete understanding of another example of an ideal $\pr_3$.  Then by Theorem 3 in \cite{wheeler/14}, a natural next step would be to begin analysis of the ideals $\pr_{n-3}$.  We anticipate the difficulty will be in studying the locally closed sets $\mathscr Y_{n,n-1,n-3}$ and $\mathscr Y_{n,n-2,n-3}$.  A possible strategy would be to apply the bundle map from Proposition \ref{prop:GrassMap}, restricted to those sets.  

\bibliographystyle{plain}
\bibliography{../wheelerbib}

\end{document}